\documentclass[11pt,a4paper]{amsart}
\usepackage[all]{xy}
\usepackage{graphicx}
\usepackage{theoremref}
\usepackage{comment}
\pdfoutput=1

\title[c-vectros and dimension vectors for cluster-finite quivers]
{c-vectors and dimension vectors for\\
cluster-finite quivers}
\author{ALFREDO N\'AJERA CH\'AVEZ}
\address {Universit\'e Paris Diderot -- Paris 7\\
          Institut de Math\'ematiques de Jussieu\\
                   UMR 7586 du CNRS\\
                   Case 7012\\
                   B\^atiment Chevaleret\\
                   75205 Paris Cedex 13\\
           France}
\email {najera@math.jussieu.fr}

\numberwithin{equation}{section}

\newcommand{\ie}{{\em i.e. }}
\newcommand{\cf}{{\em cf.}\ }

\newtheorem{theorem}{Theorem}

\newtheorem{proposition}[theorem]{Proposition}
\newtheorem{corollary}[theorem]{Corollary}

\newtheorem{lemma and definition}[theorem]{Lemma and Definition}

\theoremstyle{definition}

\newtheorem{notation}[theorem]{Notation}
\newtheorem{remark}[theorem]{Remark}
\newtheorem{example}[theorem]{Example}
\newtheorem{definition}[theorem]{Definition}

\newcommand{\opname}[1]{\operatorname{\mathsf{#1}}}

\renewcommand{\mod}{\opname{mod}\nolimits}

\newcommand{\op}{^{op}}

\newcommand{\dimv}{\underline{\dim}\,}

\newcommand{\Z}{\mathbb{Z}}
\newcommand{\T}{\mathbb{T}}

\newcommand{\Hom}{\opname{Hom}}
\newcommand{\End}{\opname{End}}

\newcommand{\Ext}{\opname{Ext}}

\newcommand{\cc}{{\mathcal C}}
\newcommand{\cd}{{\mathcal D}}

\begin{document}

\maketitle

\begin{abstract}
Let $(Q,W)$ be a quiver with a non degenerate potential. We give a new description of the \textbf{c}-vectors of $Q$. We use it to show that, if $Q$ is mutation equivalent to a Dynkin quiver, then the set of positive $\mathbf{c}$-vectors of the cluster algebra associated to $Q^{\text{op}}$ coincides with the set of dimension vectors of the indecomposable modules over the Jacobian algebra of $(Q,W)$.
\end{abstract}

\section{Introduction}

Since the introduction of cluster categories in  \cite{BMRRT}, \cf also \cite{CCS}, it has been an interesting problem to find representation theoretic lifts of the concepts and constructions surrounding cluster algebras. In this note, we focus our attention on the \textbf{c}-vectors associated to a quiver and continue the approach started in \cite{Nájera}. We recall a fundamental fact first proved in \cite{QP 2}, namely, each \textbf{c}-vector is non-zero and all its entries are either  nonpositive or nonnegative.  From \cite[Section 8]{Nagao} (see also \cite[Theorem 4]{Nájera}) we know that positive \textbf{c}-vectors associated to a quiver $Q$ are always dimension vectors of indecomposable rigid modules over an appropriate algebra, namely, the Jacobian algebra $J(Q,W)$ associated to $Q^{\op}$ and a generic potential $W^{op}$ on it. We denote by $\Phi^{re,Sch}(Q,W)$ the set of these dimension vectors and call its elements the \emph{real Schur roots} associated to $(Q,W)$.

Let $\underline{\mathbf{c}}_+(Q)$ be the set of positive \textbf{c}-vectors associated to $Q$. We are interested in determining when the following equality holds
\begin{equation} \label{main equality}
\underline{\mathbf{c}}_+(Q)=\Phi^{re,Sch}(Q,W).
\end{equation}
A first step in this direction was accomplished in \cite{Nájera} where we proved that \eqref{main equality} holds if $Q$ is acyclic. In this note, we introduce the so-called  \textbf{c}-modules (see \thref{c-modules}) which allow us to deduce \eqref{main equality}  for all cluster-finite (skew-symmetric) cluster algebras. This qualitative result complements the work of T. Nakanishi and S. Stella \cite{Nakanishi-Stella}, who gave an explicit diagrammatic description of the \textbf{c}-vectors for cluster-finite (skew-symmetrizable) cluster algebras.

\section{\textbf{c}-modules}
Throughout this note, we will freely use the concepts from the study of cluster categories. We refer  the the reader to \cite{QP 1} for details concerning quivers with potentials (QP's for short), to \cite{Amiot,BMRRT,Plamondon} for the definition and further properties of cluster categories and to \cite{Keller-Yang} \cite{Nagao} and section 7 of \cite{Keller CD} for the construction of distinguished triangulated equivalences induced by QP mutation.

\begin{notation}
Let $(Q,W)$ be a QP with a non degenerate potential $W$. Let $n$ be the number of vertices of $Q$. We denote by $\Gamma$ (resp. $\cc$) the associated Ginzburg dg-algebra (resp. cluster category). Whenever there are unspecified morphisms or indices, we assume they are the ``most natural" ones, for instance, in the expression 
\begin{equation*}
\bigoplus_{i \rightarrow j} \Gamma_i \longrightarrow \Gamma_j
\end{equation*}
the sum is taken over all arrows in $Q$ ending at $j$ and the morphism is given by left multiplication by the arrows. In particular, associated to every cluster tilting object $T'=\bigoplus T'_i$ in $\cc$, we have exchange triangles 
\begin{equation*}
\xymatrix{\displaystyle
T'^{\ast}_j  \longrightarrow \bigoplus_{i \rightarrow j} T'_i \longrightarrow T'_j \longrightarrow  \Sigma T'^{\ast}_j
& \text{and} &\displaystyle T'_j \longrightarrow \bigoplus_{j \rightarrow i} T'_i \longrightarrow T'^{\ast}_j \longrightarrow  \Sigma T'_j.
}
\end{equation*}
When there is no risk of confusion, we will abbreviate the expression $\Hom (X,Y)$ by $(X,Y)$. Finally, for each $i\in\Z$, denote by pr$\Sigma^i\Gamma$ the full subcategory of $\cd \Gamma$ whose objects are the cones of morphisms in add$\Sigma^i \Gamma$. 
\end{notation}

\begin{theorem}
\thlabel{c-modules}
Let $t$ be a vertex of the $n$-regular tree $\T_n $ and $T'=\bigoplus T'_i$ be the corresponding cluster-tilting object in $\cc$, \cf section 7.7 of \cite{Keller CD}. If $\underline{\mathbf{c}}_j(t)$ is positive, then it equals the dimension vector of the $J(Q,W)$-module 
\begin{equation}
\opname{coker}\left(\bigoplus_{j\rightarrow i}\Hom(T,\Sigma T'_i) \longrightarrow \Hom(T,\Sigma T'^{\ast}_j) \right).
\end{equation} \label{positive c-modules}
If $\underline{\mathbf{c}}_j(t)$ is negative, then its opposite equals the dimension vector of the $J(Q,W)$-module 
\begin{equation} \label{negative c-modules}
\opname{coker}\left(\bigoplus_{i\rightarrow j} \Hom(T,\Sigma T'_i) \longrightarrow \Hom(T,\Sigma T'_j) \right).
\end{equation} 
Moreover, for each pair $(t,j)$ exactly one of the modules described above is nonzero.
\end{theorem}
For obvious reasons we call modules of this form \textbf{c}-modules. \thref{c-modules} is a direct consequence of the following proposition and the fact that $H^{\ast}(S_j(t))$ is concentrated either in degree $0$ or degree $1$. Let $(Q',W')$ be a QP which is mutation equivalent to $(Q,W)$. Denote by $\Gamma'$ the corresponding Ginzburg dg algebra and by $\cc'$ the corresponding cluster category.

\begin{proposition}
\thlabel{homology as cokernel}
Let $F:\cd \Gamma' \tilde{\longrightarrow} \cd \Gamma$ be a triangle equivalence such that  $ \cd_{\leq 0}\Gamma \subseteq F\cd_{\leq 0}\Gamma' \subseteq \cd_{\leq 1}\Gamma$ and that $F(\Gamma')$ lies in pr$\Sigma^{-1}\Gamma$. Suppose that $T'$, the image of $F(\Gamma')$ in $\cc$, is a cluster tilting object. Then there are short exact sequences
\begin{equation*}
\xymatrix{\displaystyle
\bigoplus_{j\rightarrow i} \Hom_{\cc}(T, \Sigma T'_i)\ar[r] & \Hom_{\cc}(T,\Sigma T'^{\ast}_j)\ar[r] & H^0(FS'_j)\ar[r] & 0\phantom{.}
}
\end{equation*}
and
\begin{equation*}
\xymatrix{
 \displaystyle{\bigoplus_{i\rightarrow j}} \Hom_{\cc}(T, \Sigma T'_i)\ar[r] &  \Hom_{\cc}(T,\Sigma T'_j) \ar[r] & H^1(FS_j)\ar[r] & 0.
}
\end{equation*}
\end{proposition}
\begin{proof}
Consider the triangles
\begin{equation}
\label{triangle for U}
\xymatrix{
\displaystyle{\bigoplus_{i \rightarrow j}} \Gamma'_i \ar[r] & \Gamma'_j \ar[r] & U_j \ar[r] & \Sigma \displaystyle{\bigoplus_{i \rightarrow j}} \Gamma'_i  
}
\end{equation}
and
\begin{equation}
\label{triangle for V}
\xymatrix{
\Gamma'_j \ar[r] & \displaystyle{\bigoplus_{j\rightarrow i}} \Gamma'_i \ar[r] & V_j \ar[r] & \Sigma \Gamma'_j.
}
\end{equation}
induced by the natural morphisms. Then, there is a triangle $\Sigma V_j \longrightarrow U_j \longrightarrow S'_j \longrightarrow \Sigma^2 V_j $ in $\cd \Gamma'$. If we consider the long exact sequence in homology induced by the image of this last triangle, then, since $F\cd_{\leq 0}\Gamma' \subseteq \cd_{\leq 1}\Gamma$, we obtain
\begin{equation}
\label{H^1 FS}
H^1(FU_j)  \cong H^1(FS_j)
\end{equation}
and the exact sequence
\begin{equation} 
\label{H^0 S}
\xymatrix{
H^1(FV_j) \ar[r] & H^0(FU_j) \ar[r] & H^0(FS_j) \ar[r] & 0.
}
\end{equation}
Recall that there is an isomorphism $\Hom_{\cd \Gamma}(X, Y) \tilde{\longrightarrow} \Hom_{\cc}(X, Y)$, whenever $X$ and $Y$ lie in pr$\Sigma^i \Gamma$ (for all $i\in \Z$). Since $F\Gamma'$ and $\Gamma$ (resp. $\Sigma F\Gamma' $ and $\Gamma$) belong to pr$\Sigma^{-1}\Gamma$ (resp. pr$\Gamma$), we can apply this to the image in $\cd \Gamma$ of the triangle \eqref{triangle for U} to construct the following diagram
\begin{equation*}
\xymatrix{
*-<0pt,15pt>+<5pt,0pt>{\displaystyle{\bigoplus^{\phantom{N}}_{i \rightarrow j}}(\Gamma,  F\Gamma'_i)} \ar[r] \ar[d]^{\cong} & (\Gamma,F\Gamma'_j)  \ar[r]  \ar[d]^-{\cong} & (\Gamma,FU_j) \ar[r]\ar[d]& *-<0pt,15pt>+<5pt,0pt>{\displaystyle{\bigoplus^{\phantom{N}}_{i \rightarrow j}}(\Gamma,  \Sigma F\Gamma'_i)} \ar[r] \ar[d]^-{\cong} &  (\Gamma,\Sigma F\Gamma'_j)\phantom{.}  \ar[d]^{\cong} \\
*-<0pt,18pt>+<5pt,0pt>{\displaystyle{\bigoplus^{\phantom{N}}_{i \rightarrow j}}(T,  T'_i)} \ar[r]  & (T, T'_j)  \ar[r]  & (T,\Sigma T'^{\ast}_j) \ar[r] &*-<0pt,18pt>+<5pt,0pt>{\displaystyle{\bigoplus^{\phantom{N}}_{i \rightarrow j}}(T, \Sigma T'_i)}\ar[r] & (T,\Sigma T'_j)
.}
\vspace{2mm}
\end{equation*}
We obtain an isomorphism
\begin{equation}
\label{H^0 FU}
H^0(FU_j) \cong \Hom_{\cc}(T,\Sigma T'^{\ast}_j)
\end{equation}
and an exact sequence
\begin{equation} \label{H^1 FU 2}
\xymatrix{
\displaystyle{\bigoplus_{i \rightarrow j}} \Hom_{\cc}(T,\Sigma T'_i) \ar[r] & \Hom_{\cc}(T,\Sigma T'_j) \ar[r] & H^1(FU_j) \ar[r] & 0.
}
\end{equation}
Applying a similar argument to the triangle \eqref{triangle for V}, we obtain an epimorphism
\begin{equation}
\label{H^1 FV}
\xymatrix{
 \displaystyle{\bigoplus_{j \rightarrow i}} \Hom_{\cc}(T, \Sigma T'_i)  \ar[r] & H^1(FV_j) \ar[r] & 0.
}
\end{equation}
Combining \eqref{H^0 S} with \eqref{H^0 FU} and \eqref{H^1 FV} we obtain the natural morphism $\displaystyle \bigoplus_{j\rightarrow i} \Hom_{\cc}(T, \Sigma T'_i) \longrightarrow \Hom_{\cc}(T,\Sigma T'^{\ast}_j)$ and thus the first exact sequence of the statement. We obtain the second sequence by combining  \eqref{H^1 FS} with \eqref{H^1 FU 2}.
\end{proof}

\begin{definition}
We say that a module $M$ over a finite dimensional algebra is \emph{$\tau$-rigid} if $\Hom(M,\tau M)=0$. Notice that the images in $\mod J(Q,W)$ of the rigid indecomposable objects in $\cc$ under the canonical functor $\Hom(\Gamma,\ )$ are the $\tau$-rigid indecomposable modules, \cf section 3.5 of \cite{Keller Reiten}. 
\end{definition}

\begin{corollary}\thlabel{source}
Let $(Q,W)$ be a Jacobi-finite quiver with potential. Suppose $M$ is a $\tau$-rigid indecomposable $J(Q,W)$-module and that $M \cong \Hom_{\cc}(T, \tilde{M})$ for an object $\tilde{M}$ of $\cc$. If there is a cluster tilting object $T'$ in $\cc$ such that $\tilde{M}\cong \Sigma T'_j$ where $T'_j$ is a source of $T'$, then $\dimv(M)$ is a $\mathbf{c}$-vector.
\end{corollary}

\begin{theorem} \thlabel{cluster-finite case}
Let $(Q,W)$ be a quiver with potential which is mutation equivalent to a Dynkin quiver $(\vec{\Delta},0)$ (i.e.~$Q$ is cluster-finite). Then the set of positive $\mathbf{c}$-vectors equals the set of dimension vectors of the indecomposable $J(Q,W)$-modules.
\end{theorem}

\begin{proof}
Let $M$ be an indecomposable $J(Q,W)$-module. We can find an indecomposable object $X$ in $\cc_{Q.W}$ such that $\Hom_{\cc}(T,X)=M$. Since $\cc_{Q,W}$ is triangle equivalent to the usual cluster category $\cc_{\vec{\Delta}}$, the object $X$ must be rigid and thus $M$ is $\tau$-rigid (in particular rigid). It is easy to see that for Dynkin quivers we can obtain every rigid indecomposable object in the cluster category as source of a cluster tilting object (see the proof of Proposition 3 in \cite{Nájera}). Now, using Corollary \ref{source} we obtain the result.
\end{proof}

\begin{remark}
It follows that the sets of \textbf{c}-vectors of $Q$  and $Q^{\text{op}}$ coincide. Indeed, the sets of dimension vectors of indecomposable modules over $J(Q,W)$ and $J(Q,W)^{\text{op}}$ coincide.
\end{remark}

\section{examples}

We analyze two examples, one of them beyond cluster-finite type, to show how \thref{c-modules} can be used to understand families of \textbf{c}-vectors.

\begin{remark}
\thlabel{H and F morphisms}
(\cite{BMRRT}) Let $Q$ be a quiver mutation equivalent to an acyclic quiver $Q'$ and $k$ be a field. Recall that $\Hom_{\cc}(X,Y)= \Hom_{kQ'}(X,Y)\oplus \Ext^1_{kQ'}(X,\tau^{-1}Y)$ for any pair of objects $X$ and $ Y$ in $\cc$ and that only one of these summands is non-zero. We call the elements of $\Hom_{\cd^b(Q')}(X,Y)$ (resp. $\Ext^1_{kQ'}(X,\tau^{-1}Y)$) $H$-morphisms (resp. $F$-morphisms). 
\end{remark}

\begin{definition}
Let $U$ be a rigid indecomposable object of $\cc$ and $T'$ a reachable cluster-tilting object containing $U$ as a direct summand. Let $t$ be the vertex of $\T_n$ corresponding to $T'$ and $j$ the vertex of $Q$ corresponding to $U$. We write $\underline{\mathbf{c}}_{U,T'}$ for the \textbf{c}-vector $\underline{\mathbf{c}}_j(t)$. Denote by $\underline{\mathbf{c}}_U$ the set of vectors obtained in this way and by $\underline{\mathbf{c}}^-_U$ (resp. $\underline{\mathbf{c}}^+_U$) its subset of negative (resp. positive) vectors. Note that this notation depends on the orientation of Q. 
\end{definition}

\begin{example} (Type $\tilde{A}_{1,1}$)
Let $\cc$ be a cluster category of type $\tilde{A}_{1,1}$, \ie the cluster category associated with the quiver

\begin{equation*}
Q: \xymatrix{
&2 \ar[dr]& \\
1\ar[ru] \ar[rr]& &3.
}
\end{equation*}
Here, we let $T_i =P_i$, \ie the image in $\cc$ of the indecomposable projective $kQ$-module associated to the vertex $i$. We take $T=T_1\oplus T_2 \oplus T_3$ as the initial cluster tilting object and let $T'=\Sigma^{-1} T_1\oplus S_2 \oplus \Sigma^{-1} T_3$. We can visualize $\End_{\cc}(T')$ as follows 
\begin{equation*}
\xymatrix{
&S_2 \ar[dl]& \\
\Sigma^{-1}T_1\ar@<0.5ex>[rr] \ar@<-0.5ex>[rr]& &\Sigma^{-1}T_3. \ar[lu]
}
\end{equation*}
The cluster-tilting objects of the cluster category $\cc$ are in bijection with the vertices of the planar graph in Figure 1 and the indecomposable rigid objects with the connected components of its complement in the plane, \cf \cite{Clusters 2} \cite{Cerulli-Irelli}.
\begin{figure}[htbp]
\includegraphics[scale=0.65]{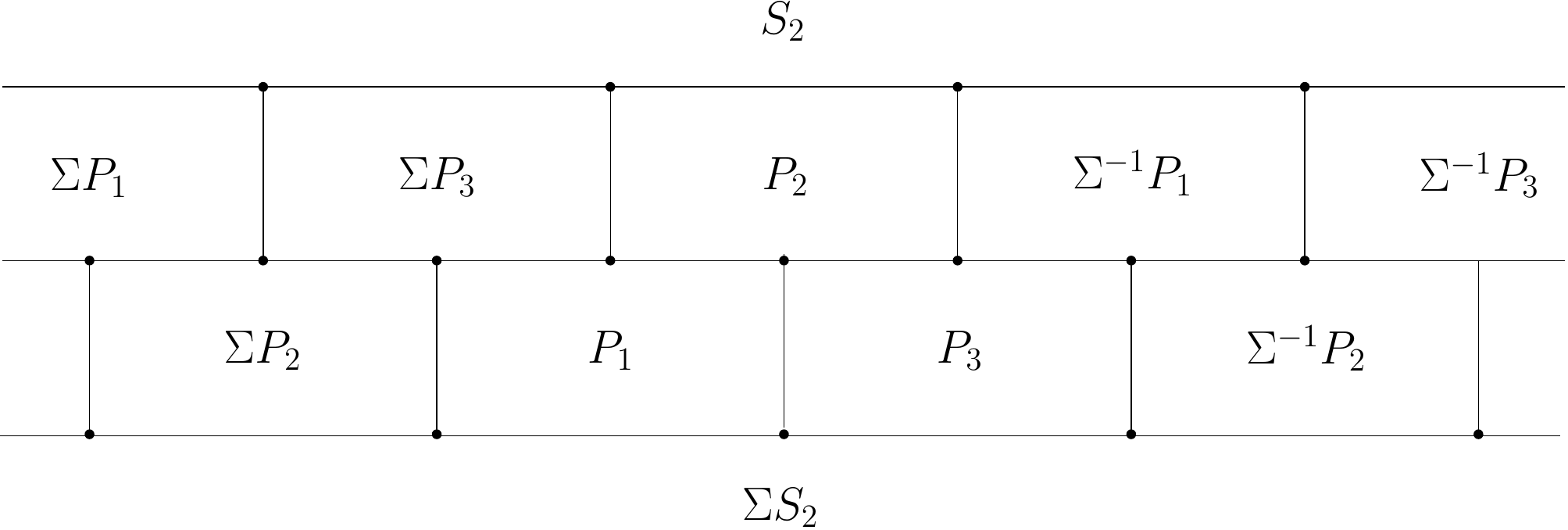}
\caption{Exchange graph of type $\tilde{A}_{1,1}$.}
\end{figure}
Let us compute $\underline{\mathbf{c}}_{ \Sigma^{-1} T_3,T'}$ using the formulas in \thref{c-modules}. Since $T$ is the image of a projective $kQ$-module, we have $\Hom_{\cc}(T,X)\cong \Hom_{kQ}(T,X)$ for every $kQ$-module $X$ (see \thref{H and F morphisms}). Therefore, the morphism in the equation \eqref{negative c-modules} becomes
\begin{equation*} 
\Hom_{kQ}(T, T_1)\oplus\Hom_{kQ}(T,T_1) \longrightarrow \Hom_{kQ}(T,T_3).
\end{equation*} 
Its cokernel is the indecomposable $kQ$-module with dimension vector $(0,1,1)$, thus $\underline{\mathbf{c}}_{ \tau^{-1} T_3,T'}=(0,-1,-1)$.

Now suppose our initial seed is given by the cluster tilting object $T=P_1\oplus \tau S_2 \oplus P_3$. We proceed to compute the family $ \underline{\mathbf{c}}_{S_2} $.  Notice that the complements of $S_2$ are of the form:

\begin{equation*}
\begin{xy} 0;<0.5pt,0pt>:<0pt,-0.5pt>::
(-10,-10) *+{(A_i)},
(265,-10) *+{(B_i)},
(540,-10) *+{(C_i)},
(0,75) *+{\Sigma^iP_1} ="0",
(75,0) *+{S_2} ="1",
(150,75) *+{\Sigma^iP_3} ="2",
(275,75) *+{\Sigma^{i+1}P_2} ="3",
(350,0) *+{S_2} ="4",
(425,75) *+{\Sigma^iP_1} ="5",
(550,75) *+{\Sigma^iP_3} ="6",
(625,0) *+{S_2} ="7",
(700,75) *+{\Sigma^{i-1}P_2,} ="8",
"1", {\ar"0"},
"0", {\ar@<0.5ex>"2"},
"0", {\ar@<-0.5ex>"2"},
"2", {\ar"1"},
"4", {\ar"3"},
"3", {\ar@<0.5ex>"5"},
"3", {\ar@<-0.5ex>"5"},
"5", {\ar"4"},
"7", {\ar"6"},
"6", {\ar@<0.5ex>"8"},
"6", {\ar@<-0.5ex>"8"},
"8", {\ar"7"},
\end{xy}
\end{equation*}
where $i$ is an odd integer. First suppose $T'$ is as in cases $(A_i), (B_i)$ or $(C_i)$ for $i\geq 1$. Note that there are only $H$-morphisms from $T$ to $\Sigma S_2$. Suppose that the morphism in \eqref{negative c-modules} is non-zero for any of the cases mentioned above. Then there exists a commutative diagram
\begin{equation*}
\xymatrix{
&X \ar^h[dr]& \\
T \ar^g[ru] \ar^f[rr]& &\Sigma S_2,
}
\end{equation*}
for every $X\in \lbrace \Sigma^iP_1, \Sigma^{i-1}P_2, \Sigma^iP_3:i\geq 0\rbrace$. By \thref{H and F morphisms} one of $g$ or $h$ is an $F$-morphism. This implies that $f$ is an $F$-morphism which is a contradiction. Thus $\underline{\mathbf{c}}_{S_2,T'}=(-1,-1,-1)$. It is easy to see that in the rest of the cases, a factorization will occur. We obtain in this way $\underline{\mathbf{c}}_{S_2}=\lbrace (-1,-1,-1), (0,-1,-1),(0,-1,0)\rbrace$.


\end{example}

Let $\cc$ be a cluster category with finitely many isomorphism classes of indecomposables. By \thref{cluster-finite case} we know that the (positive) \textbf{c}-vectors are of the form $ \dimv\Hom_{\cc}(T,U)$ where $U$ runs through the indecomposable objects of $\cc$. This already has some nice consequences. For example, we claim that the components of the \textbf{c}-vectors of a cluster-finite quiver $Q$ are bounded by $6$ and that this bound is reached only for quivers of cluster type $E_8$. Indeed, by \thref{H and F morphisms}, it follows that the components are bounded by the maximum of the $\dim \Hom_{\cd^b(Q)}$(L.M), where $L$ and $M$ run through the indecomposable objects of $\cd^{b}(Q)$. Now by Happel's description of $\cd^b(Q)$ \cite{Happel}, these dimensions equal the coefficients of the simple roots in the decomposition of the positive roots in the root system of type $\Delta$, where $\Delta$ is the cluster type of $Q$. It follows that $6$ is indeed the upper bound and that it is reached  only for type $\Delta = E_8$. Let us determine the structure of the quivers of cluster type $E_8$ admitting a \textbf{c}-vector with a component equal to $6$.

\begin{example} (Type $E_8$)
Let $\cc$ be a cluster category of type $E_8$. We know that the coefficients of the positive roots (expressed in terms of the simple roots) in an $E_8$ root system are bounded by $6$. Moreover, 6 can only appear as a coefficient of the simple root associated to the vertex of valency $3$ in $E_8$. We claim that the cluster-finite quivers admitting a \textbf{c}-vector with an entry equal to 6 are those obtained by gluing quivers of cluster type $A_2,\  A_3$ and $A_5$ in a common vertex. To see this, we can use Figure 2, which depicts the AR-quiver of an $E_8$ quiver. The vertices represented by $\Box$ are the modules $Y$ for which $\dim\Hom_{kE_8}(X,Y)=6$. This implies that if a \textbf{c}-vector $\dimv \Hom(T,U)$ associated with the quiver $Q$ of $\End(T)$ has a component equal to $6$, then $T$ contains an indecomposable factor $T_1$ in the orbit of $X$. The other factors $T_j$ satisfy $\Ext^1_{\cc}(T_1,T_j)=0$. Thus, if $T_1$ corresponds to $\ast$ in Figure 2, the factors $T_j$ correspond to vertices $\circ$. We know from \cite{Iyama-Yoshino} that the possible complements of $T_1$ are in bijection with the cluster-tilting objects in the Calabi-Yau reduction $T_1^{\perp}/\langle T_1 \rangle$, where
\begin{equation*}
T_1^{\perp}=\lbrace X\in \cc | \Ext^1(T,X)=0\rbrace
\end{equation*}
and $\langle T_1 \rangle$ is the ideal of morphisms factoring through a sum of copies of $T_1$. We see that $T_1^{\perp}/\langle T_1 \rangle$ is equivalent to the product $\cc_{A_1}\times \cc_{A_2}\times \cc_{A_4} $. The claim follows.
\end{example}

\begin{figure}[htbp!]\label{AR quiver}
\begin{xy} 0;<1pt,0pt>:<0pt,-0.8pt>:: 
(-13,0) *+{\bullet} ="1",
(-13,112) *+{X} ="2",
(-13,196) *+{\bullet} ="3",
(-13,280) *+{\bullet} ="4",
(1,42) *+{\bullet} ="5",
(1,85) *+{\bullet} ="6",
(1,154) *+{\bullet} ="7",
(1,238) *+{\bullet} ="8",
(15,0) *+{\bullet} ="9",
(15,112) *+{\bullet} ="10",
(15,196) *+{\bullet} ="11",
(15,280) *+{\circ} ="12",
(29,42) *+{\bullet} ="13",
(29,85) *+{\bullet} ="14",
(29,154) *+{\bullet} ="15",
(29,238) *+{\circ} ="16",
(43,0) *+{\circ} ="17",
(43,112) *+{\bullet} ="18",
(43,196) *+{\circ} ="19",
(43,280) *+{\circ} ="20",
(57,42) *+{\circ} ="21",
(57,85) *+{\circ} ="22",
(57,154) *+{\circ} ="23",
(57,238) *+{\circ} ="24",
(71,0) *+{\circ} ="25",
(71,112) *+{\star} ="26",
(71,196) *+{\circ} ="27",
(71,280) *+{\circ} ="28",
(85,42) *+{\circ} ="29",
(85,85) *+{\circ} ="30",
(85,154) *+{\circ} ="31",
(85,238) *+{\circ} ="32",
(99,0) *+{\circ} ="33",
(99,112) *+{\bullet} ="34",
(99,196) *+{\circ} ="35",
(99,280) *+{\circ} ="36",
(113,42) *+{\bullet} ="37",
(113,85) *+{\bullet} ="38",
(113,154) *+{\bullet} ="39",
(113,238) *+{\circ} ="40",
(127,0) *+{\bullet} ="41",
(127,112) *+{\Box} ="42",
(127,196) *+{\bullet} ="43",
(127,280) *+{\circ} ="44",
(141,42) *+{\bullet} ="45",
(141,85) *+{\bullet} ="46",
(141,154) *+{\bullet} ="47",
(141,238) *+{\bullet} ="48",
(155,0) *+{\bullet} ="49",
(155,112) *+{\Box} ="50",
(155,196) *+{\bullet} ="51",
(155,280) *+{\bullet} ="52",
(169,42) *+{\bullet} ="53",
(169,85) *+{\bullet} ="54",
(169,154) *+{\bullet} ="55",
(169,238) *+{\bullet} ="56",
(183,0) *+{\bullet} ="57",
(183,112) *+{\Box} ="58",
(183,196) *+{\bullet} ="59",
(183,280) *+{\bullet} ="60",
(197,42) *+{\bullet} ="61",
(197,85) *+{\bullet} ="62",
(197,154) *+{\bullet} ="63",
(197,238) *+{\bullet} ="64",
(211,0) *+{\bullet} ="65",
(211,112) *+{\Box} ="66",
(211,196) *+{\bullet} ="67",
(211,280) *+{\bullet} ="68",
(225,42) *+{\bullet} ="69",
(225,85) *+{\bullet} ="70",
(225,154) *+{\bullet} ="71",
(225,238) *+{\bullet} ="72",
(239,0) *+{\bullet} ="73",
(239,112) *+{\Box} ="74",
(239,196) *+{\bullet} ="75",
(239,280) *+{\bullet} ="76",
(253,42) *+{\bullet} ="77",
(253,85) *+{\bullet} ="78",
(253,154) *+{\bullet} ="79",
(253,238) *+{\bullet} ="80",
(267,0) *+{\bullet} ="81",
(267,112) *+{\bullet} ="82",
(267,196) *+{\bullet} ="83",
(267,280) *+{\bullet} ="84",
(281,42) *+{\bullet} ="85",
(281,85) *+{\bullet} ="86",
(281,154) *+{\bullet} ="87",
(281,238) *+{\bullet} ="88",
(295,0) *+{\bullet} ="89",
(295,112) *+{\bullet} ="90",
(295,196) *+{\bullet} ="91",
(295,280) *+{\bullet} ="92",
(309,42) *+{\bullet} ="93",
(309,85) *+{\bullet} ="94",
(309,154) *+{\bullet} ="95",
(309,238) *+{\bullet} ="96",
(323,0) *+{\bullet} ="97",
(323,112) *+{\bullet} ="98",
(323,196) *+{\bullet} ="99",
(323,280) *+{\bullet} ="100",
(337,42) *+{\bullet} ="101",
(337,85) *+{\bullet} ="102",
(337,154) *+{\bullet} ="103",
(337,238) *+{\bullet} ="104",
(351,0) *+{\bullet} ="105",
(351,112) *+{\bullet} ="106",
(351,196) *+{\bullet} ="107",
(351,280) *+{\bullet} ="108",
(365,42) *+{\bullet} ="109",
(365,85) *+{\bullet} ="110",
(365,154) *+{\bullet} ="111",
(365,238) *+{\bullet} ="112",
(379,0) *+{\bullet} ="113",
(379,112) *+{\bullet} ="114",
(379,196) *+{\bullet} ="115",
(379,280) *+{\bullet} ="116",
(393,42) *+{\bullet} ="117",
(393,85) *+{\bullet} ="118",
(393,154) *+{\bullet} ="119",
(393,238) *+{\bullet} ="120",
"1", {\ar"5"},
"2", {\ar"5"},
"2", {\ar"6"},
"2", {\ar"7"},
"3", {\ar"7"},
"3", {\ar"8"},
"4", {\ar"8"},
"5", {\ar"9"},
"5", {\ar"10"},
"6", {\ar"10"},
"7", {\ar"10"},
"7", {\ar"11"},
"8", {\ar"11"},
"8", {\ar"12"},
"9", {\ar"13"},
"10", {\ar"13"},
"10", {\ar"14"},
"10", {\ar"15"},
"11", {\ar"15"},
"11", {\ar"16"},
"12", {\ar"16"},
"13", {\ar"17"},
"13", {\ar"18"},
"14", {\ar"18"},
"15", {\ar"18"},
"15", {\ar"19"},
"16", {\ar"19"},
"16", {\ar"20"},
"17", {\ar"21"},
"18", {\ar"21"},
"18", {\ar"22"},
"18", {\ar"23"},
"19", {\ar"23"},
"19", {\ar"24"},
"20", {\ar"24"},
"21", {\ar"25"},
"21", {\ar"26"},
"22", {\ar"26"},
"23", {\ar"26"},
"23", {\ar"27"},
"24", {\ar"27"},
"24", {\ar"28"},
"25", {\ar"29"},
"26", {\ar"29"},
"26", {\ar"30"},
"26", {\ar"31"},
"27", {\ar"31"},
"27", {\ar"32"},
"28", {\ar"32"},
"29", {\ar"33"},
"29", {\ar"34"},
"30", {\ar"34"},
"31", {\ar"34"},
"31", {\ar"35"},
"32", {\ar"35"},
"32", {\ar"36"},
"33", {\ar"37"},
"34", {\ar"37"},
"34", {\ar"38"},
"34", {\ar"39"},
"35", {\ar"39"},
"35", {\ar"40"},
"36", {\ar"40"},
"37", {\ar"41"},
"37", {\ar"42"},
"38", {\ar"42"},
"39", {\ar"42"},
"39", {\ar"43"},
"40", {\ar"43"},
"40", {\ar"44"},
"41", {\ar"45"},
"42", {\ar"45"},
"42", {\ar"46"},
"42", {\ar"47"},
"43", {\ar"47"},
"43", {\ar"48"},
"44", {\ar"48"},
"45", {\ar"49"},
"45", {\ar"50"},
"46", {\ar"50"},
"47", {\ar"50"},
"47", {\ar"51"},
"48", {\ar"51"},
"48", {\ar"52"},
"49", {\ar"53"},
"50", {\ar"53"},
"50", {\ar"54"},
"50", {\ar"55"},
"51", {\ar"55"},
"51", {\ar"56"},
"52", {\ar"56"},
"53", {\ar"57"},
"53", {\ar"58"},
"54", {\ar"58"},
"55", {\ar"58"},
"55", {\ar"59"},
"56", {\ar"59"},
"56", {\ar"60"},
"57", {\ar"61"},
"58", {\ar"61"},
"58", {\ar"62"},
"58", {\ar"63"},
"59", {\ar"63"},
"59", {\ar"64"},
"60", {\ar"64"},
"61", {\ar"65"},
"61", {\ar"66"},
"62", {\ar"66"},
"63", {\ar"66"},
"63", {\ar"67"},
"64", {\ar"67"},
"64", {\ar"68"},
"65", {\ar"69"},
"66", {\ar"69"},
"66", {\ar"70"},
"66", {\ar"71"},
"67", {\ar"71"},
"67", {\ar"72"},
"68", {\ar"72"},
"69", {\ar"73"},
"69", {\ar"74"},
"70", {\ar"74"},
"71", {\ar"74"},
"71", {\ar"75"},
"72", {\ar"75"},
"72", {\ar"76"},
"73", {\ar"77"},
"74", {\ar"77"},
"74", {\ar"78"},
"74", {\ar"79"},
"75", {\ar"79"},
"75", {\ar"80"},
"76", {\ar"80"},
"77", {\ar"81"},
"77", {\ar"82"},
"78", {\ar"82"},
"79", {\ar"82"},
"79", {\ar"83"},
"80", {\ar"83"},
"80", {\ar"84"},
"81", {\ar"85"},
"82", {\ar"85"},
"82", {\ar"86"},
"82", {\ar"87"},
"83", {\ar"87"},
"83", {\ar"88"},
"84", {\ar"88"},
"85", {\ar"89"},
"85", {\ar"90"},
"86", {\ar"90"},
"87", {\ar"90"},
"87", {\ar"91"},
"88", {\ar"91"},
"88", {\ar"92"},
"89", {\ar"93"},
"90", {\ar"93"},
"90", {\ar"94"},
"90", {\ar"95"},
"91", {\ar"95"},
"91", {\ar"96"},
"92", {\ar"96"},
"93", {\ar"97"},
"93", {\ar"98"},
"94", {\ar"98"},
"95", {\ar"98"},
"95", {\ar"99"},
"96", {\ar"99"},
"96", {\ar"100"},
"97", {\ar"101"},
"98", {\ar"101"},
"98", {\ar"102"},
"98", {\ar"103"},
"99", {\ar"103"},
"99", {\ar"104"},
"100", {\ar"104"},
"101", {\ar"105"},
"101", {\ar"106"},
"102", {\ar"106"},
"103", {\ar"106"},
"103", {\ar"107"},
"104", {\ar"107"},
"104", {\ar"108"},
"105", {\ar"109"},
"106", {\ar"109"},
"106", {\ar"110"},
"106", {\ar"111"},
"107", {\ar"111"},
"107", {\ar"112"},
"108", {\ar"112"},
"109", {\ar"113"},
"109", {\ar"114"},
"110", {\ar"114"},
"111", {\ar"114"},
"111", {\ar"115"},
"112", {\ar"115"},
"112", {\ar"116"},
"113", {\ar"117"},
"114", {\ar"117"},
"114", {\ar"118"},
"114", {\ar"119"},
"115", {\ar"119"},
"115", {\ar"120"},
"116", {\ar"120"},
\end{xy}
\caption{The AR-quiver of an $E_8$ quiver.}
\end{figure}


\begin{thebibliography}{1}

\bibitem{Amiot}
C. Amiot, \emph{Cluster categories for algebras of global dimension 2 and quivers with potential}, Annales de l'Institut Fourier \textbf{59} (2009), no. 6, 2525-2590.

\bibitem{BMRRT}
A. B. Buan, R. J. Marsh, M. Reineke, I. Reiten and G. Todorov, \emph{Tilting theory and cluster combinatorics}, Advances in Mathematics \textbf{204 (2)} (2006), 572--618.

\bibitem{BMR}
A. B. Buan, R. Marsh and I. Reiten, \emph{Denominators of cluster variables}, J. London Math. Soc. \textbf{79}, no. 3 (2009). 589--611.

\bibitem{CCS}
P. Caldero, F. Chapoton and R. Schiffler, \emph{Quivers with relations arising from clusters
($A_n$-case)}. Trans. Amer. Math. Soc. \textbf{358} (2006), 1347--1364.

\bibitem{Cerulli-Irelli}
G. Cerulli Irelli, \emph{Cluster algebras of type $A_2^{(1)}$} Algebras and Representation Theory, \emph{15} (2012), 977--1021. 

\bibitem{QP 1}
H. Derksen, J. Weyman and A. Zelevinsky, \emph{Quivers with potentials and their representations I: Mutations}, Selecta Math. \textbf{14} (2008), no. 1, 59--119.

\bibitem{QP 2}
\bysame, \emph{Quivers with potentials and their representations II: Applications to cluster algebras}, J. Amer. Math. Soc. $\mathbf{23}$ (2010), 749--790.

\bibitem{Clusters 2}
S. Fomin and A. Zelevinsky, \emph{Cluster algebras II: Finite type classification}, Invent. Math. \textbf{154}
(2003), 63--121.

\bibitem{Ginzburg}
V. Ginzburg, \emph{Calabi-Yau algebras}, arXiv:math/0612139v3 [math.AG].

\bibitem{Happel}
D. Happel, \emph{On the derived category of a finite-dimensional algebra}, Comment. Math. Helv, \textbf{62}(1987), no. 3. 339--389.

\bibitem{Iyama-Yoshino}
O. Iyama and Y. Yoshino, \emph{Mutation in triangulated categories and rigid Cohen-Macaulay
modules}. Invent. Math. \emph{172}, no. 1 (2008): 117--168.

\bibitem{Keller CD}
B. Keller, \emph{Cluster algebras and derived categories}, eprint, arXiv:1202.4161v4 [math.RT].

\bibitem{Keller Reiten}
B. Keller, I. Reiten, \emph{Cluster-tilted algebras are Gorenstein and stably Calabi-Yau}, Advances Math. \textbf{211} (2007), 123-151.

\bibitem{Keller-Yang}
B. Keller, D. Yang, \emph{Derived equivalences from mutations of quivers with potential}, Advances in Mathematics \textbf{26} (2011), 2118--2168.

\bibitem{Nagao}
K. Nagao, \emph{Donaldson-Thomas theory and cluster algebras}, eprint, arXiv:1002.4884 [math.AG].

\bibitem{Nájera}
A. N\'ajera Ch\'avez, \emph{On the c-vectors of an acyclic cluster algebra}, eprint, arXiv:1203.1415 [math.RT].

\bibitem{Plamondon}

P. Plamondon, \emph{Cluster algebras via cluster categories with infinite-dimensional morphism spaces}, Compos. Math. \textbf{147} (2011), 1921-1954.

\bibitem{Nakanishi-Stella}
T. Nakanishi and S. Stella, \emph{Diagramatic description of c-vectors and d-vectors of cluster algebras of finite type}, eprint, arXiv:1210.6299 [math.RA].

\end{thebibliography}
\end{document}